\documentclass[11pt]{amsart}

\usepackage[a4paper]{geometry}   

\usepackage[T1]{fontenc}
\usepackage[utf8]{inputenc}

\usepackage{mathtools} 

\usepackage{hyperref} 
\usepackage{cleveref} 

\usepackage{tikz-cd} 
\usepackage{array}

\usepackage{booktabs} 

\theoremstyle{plain}

\newtheorem*{thm*}{Theorem}

\newtheorem*{prop*}{Proposition}

\newtheorem*{conj*}{Conjecture}
\newtheorem{thmintro}{Theorem} 
\newtheorem{propintro}[thmintro]{Proposition}
\newtheorem{lemmeintro}[thmintro]{Lemma}

\theoremstyle{definition}

\newtheorem*{defn*}{Definition}

\newtheorem*{question*}{Question}

\theoremstyle{remark}

\newtheorem*{rem*}{Remark}

\usepackage{pgfplots}
\usetikzlibrary{arrows}

\author{Miguel Acosta}
\author{Jean-Marc Schlenker}
\address{Department of mathematics, 
	Université du Luxembourg, 
	Maison du Nombre, 
	6, Avenue de la Fonte, 
	L-4364 Esch-sur-Alzette, 
	Luxembourg.}
\email{miguel.acosta@normalesup.org , jean-marc.schlenker@uni.lu}
\thanks{Both authors were partially supported by FNR grant AGoLoM OPEN/16/11405402} 

\title{
A hyperbolic proof of Pascal's Theorem}
\date{\today}  

\begin{document}
\maketitle


\begin{abstract}
We provide a simple proof of Pascal's Theorem on cyclic hexagons, as well as a generalization by Möbius, using hyperbolic geometry.
\end{abstract}

\section{Pascal's Theorem}

Blaise Pascal (1623--1662) is a towering intellectual figure of the XVII\textsuperscript{th} century. He is credited with inventing and building the first mechanical calculator, the \emph{Pascaline}, and with laying the foundations of probability theory, in particular in his correspondence with Fermat -- he came up for instance with Pascal's triangle. He is also known for his work on hydrostatics and Pascal's law (as well as the invention of the syringe), and for discovering the variation of air pressure with altitude -- the SI unit of pressure is called the Pascal. However, he was most influential in his time as a philosopher and theologian, and well-known for ``Pascal's wager''.

Pascal was raised and educated by his father, Etienne Pascal, who had a strong interest in the intellectual developments of his time and was an active member of a group of scientists meeting around Martin Mersenne, including Desargues, Descartes and others. According to contemporary sources \cite[p. 176]{paris1834historiettes}, Blaise was extraordinarily precocious, and so passionate in studying mathematics that, when he was 11, his father forbid him to read any mathematics book before he turned 15 and knew latin and greek. Blaise therefore continued studying geometry by himself and in secret and, at 16, published his first treatise on projective geometry, the \emph{Traité sur les coniques}, which contained the following remarkable statement.

\begin{thmintro}[Pascal's Theorem]\label{thm:pascal}
	Let $ABCDEF$ be a cyclic hexagon. Let $X$ be the intersection point of $AB$ and $DE$, $Y$ the intersection point of $BC$ and $EF$ and $Z$ the intersection point of $CD$ and $FA$.
	Then $X$, $Y$ and $Z$ are aligned.
\end{thmintro}

 The main goal of this note is to provide a simple proof, based on hyperbolic geometry, of Pascal's Theorem.
 The statement has a natural setting in the projective plane. Instead of considering a cyclic hexagon, one then considers a hexagon with vertices on a conic. 
Any non-degenerate conic is projectively equivalent to a circle, while the statements for degenerate conics can be obtained by a limiting argument where a hexagon with vertices on a degenerate conic is obtained as a limit of hexagons with vertices on non-degenerate conics.

 By considering different permutations of the set of vertices, the configuration of Pascal's theorem leads to a fascinating set of $95$ lines and $95$ points with a beautiful incidence structure. We recommend the articles of Conway and Ryba \cite{conway_pascal_2012} and \cite{conway_extending_2013} describing the full configuration, that they call \emph{mysticum hexagrammaticum}.
 On the one hand, this name describes their approach to label the lines and points by permutations of the set $\{A,B,C,D,E,F\}$, giving a natural way to state incidence relations and be able to identify symmetries.
 On the other hand, it recalls the original name of \Cref{thm:pascal}, given by Pascal in $1639$: \emph{hexagrammum mysticum}. 
 The original work from Pascal did not survive, but there are many proofs of \Cref{thm:pascal}, using a wide variety of tools.
 A proof using algebraic geometry is sketched by Conway and Ryba in \cite{conway_pascal_2012}. Other proofs involve cross-ratios and symmetries of the projective plane, as in \cite{richter-gebert_perspectives_2011} or Euclidean lengths and Menalaus's theorem like in \cite[p.77]{coxeter_geometry_1967}.

 The proof given here, based on hyperbolic geometry, is not really novel (it can be deduced easily from \cite[page 436]{richter-gebert_perspectives_2011}) but it brings to light a striking link between Pascal's theorem and elementary hyperbolic geometry. A recent and similar link is exhibited by Drach and Schwarz in \cite{drach_hyperbolic_2019}, where they revisit the seven-circles theorem, in Euclidean geometry, in terms of hyperbolic geometry.
 Even if those links between hyperbolic geometry and results on projective or Euclidean geometry do not lead to novel results, they give a beautiful perspective to them.

\section{A hyperbolic statement of Pascal's Theorem}

\begin{figure}
	\centering
		\definecolor{ffwwqq}{rgb}{1.,0.4,0.}
	\definecolor{uuuuuu}{rgb}{0.26666666666666666,0.26666666666666666,0.26666666666666666}
	\definecolor{xdxdff}{rgb}{0.49019607843137253,0.49019607843137253,1.}
	\definecolor{ududff}{rgb}{0.30196078431372547,0.30196078431372547,1.}
	
	\begin{tikzpicture}[line cap=round,line join=round,>=triangle 45,x=0.6410596026490066cm,y=0.6258620689655187cm]
	\clip(-1.5760229989258046,-3.212008370019429) rectangle (10.903315844049402,6.374768489484682);
	\draw [line width=2.pt] (4.3,0.36) ellipse (2.1020890799132075cm and 2.052255071553201cm);
	\draw [line width=2.pt,domain=-1.5760229989258046:10.903315844049402] plot(\x,{(--7.2640852513663035-0.1075620419111516*\x)/1.9456548318297386});
	\draw [line width=2.pt,domain=-1.5760229989258046:10.903315844049402] plot(\x,{(--8.868770541090159-0.8380085220898956*\x)/1.2542878957246923});
	\draw [line width=2.pt,domain=-1.5760229989258046:10.903315844049402] plot(\x,{(--15.05480042766225-1.912699216734097*\x)/0.8633209249891065});
	\draw [line width=2.pt,domain=-1.5760229989258046:10.903315844049402] plot(\x,{(--25.040932907927658-3.5566663014167315*\x)/-2.726990826090552});
	\draw [line width=2.pt,domain=-1.5760229989258046:10.903315844049402] plot(\x,{(-24.26311760914553--6.185773207979004*\x)/-1.9662938951230569});
	\draw [line width=2.pt,domain=-1.5760229989258046:10.903315844049402] plot(\x,{(--1.4282045234870022--0.22916287417287196*\x)/0.6300210686700716});
	\draw [line width=2.pt,color=ffwwqq,domain=-1.5760229989258046:10.903315844049402] plot(\x,{(--20.229281814672568--13.282843146001653*\x)/32.32755488086449});
	\draw [line width=2.pt,color=ffwwqq,domain=-1.5760229989258046:10.903315844049402] plot(\x,{(--26.671592364514925-2.2194503947361612*\x)/5.676689522244635});
	\draw [line width=2.pt,color=ffwwqq,domain=-1.5760229989258046:10.903315844049402] plot(\x,{(--246.26175552769604-37.512633539084945*\x)/20.54566170582546});
	\begin{scriptsize}
	\draw [fill=ududff] (3.5,3.54) circle (2.5pt);
	\draw[color=ududff] (3.4322414638841128,3.953280886178902) node {$A$};
	\draw [fill=xdxdff] (5.445654831829739,3.4324379580888484) circle (2.5pt);
	\draw[color=xdxdff] (5.564472868842791,3.738404853121051) node {$B$};
	\draw [fill=xdxdff] (6.699942727554431,2.594429435998953) circle (2.5pt);
	\draw[color=xdxdff] (6.820671215950229,2.895429646509483) node {$C$};
	\draw [fill=xdxdff] (7.5632636525435375,0.6817302192648558) circle (2.5pt);
	\draw[color=xdxdff] (7.845464604379981,0.7797271671706448) node {$D$};
	\draw [fill=xdxdff] (4.836272826452985,-2.8749360821518755) circle (2.5pt);
	\draw[color=xdxdff] (5.2,-3.04) node {$E$};
	\draw [fill=xdxdff] (2.8699789313299284,3.310837125827128) circle (2.5pt);
	\draw[color=xdxdff] (2.7,2.9) node {$F$};
	\draw [fill=uuuuuu] (5.882065019427815,4.406448589522651) circle (2.0pt);
	\draw[color=uuuuuu] (5.944638158098989,4.779727167170635) node {$Z$};
	\draw [fill=uuuuuu] (9.500434197236755,3.208276744646777) circle (2.0pt);
	\draw[color=uuuuuu] (9.680175348181635,2.961545348988822) node {$X$};
	\draw [fill=uuuuuu] (2.126393451783468,5.650086181438344) circle (2.0pt);
	\draw[color=uuuuuu] (2.2421588192560136,5.9202230349392275) node {$Y$};
	\end{scriptsize}
	\end{tikzpicture}
	\caption{The configuration of Pascal's theorem (\Cref{thm:pascal}) and \Cref{prop:hexagon_orthogonals}.}
\end{figure}
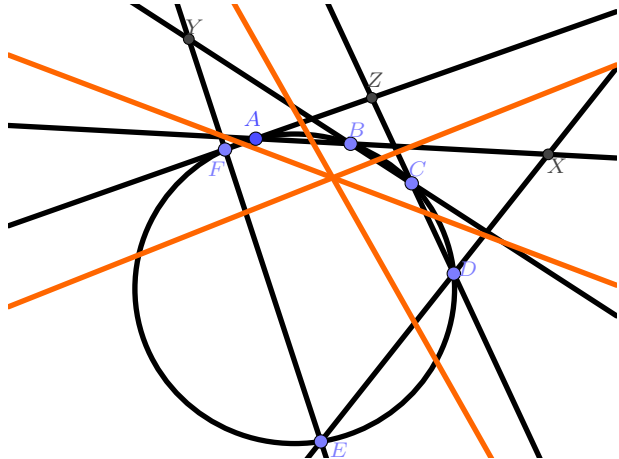

We are going to use the Klein model of the hyperbolic plane.
Consider an open disk $\Delta$ on the projective plane, bounded by a circle $\Gamma$.
The \emph{hyperbolic plane} is defined as $\Delta$, endowed with the
\emph{Hilbert distance}, defined as follows
\begin{defn*}
	Let $P,Q \in \Delta$, and $A,B$ be the intersection points of the line $PQ$ with $\Gamma$. Then
	\[ d(P,Q) = \frac{1}{2} \log \left(\frac{BP}{BQ} \cdot \frac{AQ}{AP}\right) .\]
\end{defn*}
This distance induces a complete Riemannian metric on $\Delta$, and thus notions of angles and lengths. A key property of this Hilbert distance is that it is invariant under projective transformations that leave $\Delta$ invariant. The geodesics are precisely the straight lines, but the angles are not the Euclidean ones. The circle $\Gamma$ is then the \emph{boundary at infinity of the hyperbolic plane}, and its points are called \emph{ideal points}.
The model is very rich, but we focus here only on a few points related to the polarity in the projective plane and orthogonality in the hyperbolic plane. 

Given a circle in the projective plane, one can define a polarity relation between points and lines: there is a \emph{polar line} for each point, and a \emph{pole} for each line. By definition, given points $P \in \Delta$ and $Q\not \in \bar\Delta$, $Q$ is in the polar line of $P$, and conversely, if and only if
$$ \frac{BP}{BQ} \cdot \frac{AQ}{AP} =-1~, $$
where $A$ and $B$ are again the intersections of the line $PQ$ with $\Gamma$. We will denote the polar line of $P$ by $P^*$. The polarity relation is also invariant under projective transformations that leave $\Delta$ invariant.

It follows from the definition that if $P,Q\in \Delta$, then $P$ and $Q$ are both in the polar line of $P^*\cap Q^*$, the intersection point of the polar lines $P^*$ and $Q^*$. It follows that the line $PQ$ is the polar line of $P^*\cap Q^*$. As a consequence, the polarity relation preserves the incidence: three points are aligned if and only if their polar lines are concurrent.

 This last fact allows to have \emph{dual} statements.
 For example, the dual statement of Pascal's theorem is Brianchon's theorem: if a conic is inscribed in a hexagon, then the three diagonals joining the opposite vertices of the hexagon are concurrent.

 Back to hyperbolic geometry, the link between polarity and the Klein model that we are going to use is the following:
\begin{prop*}
	Two lines $l_1$ and $l_2$ in the hyperbolic plane are orthogonal if and only if the pole of $l_1$ is contained in the extension of $l_2$ to the projective plane.
\end{prop*}
In particular, if $l_1$ and $l_2$ are two lines that do not intersect in the hyperbolic plane, they have a unique common perpendicular, that is the polar line of their intersection point in the projective plane.

The proposition follows from the projective invariance of the hyperbolic metric and of the polarity relation under projective transformations leaving $\Delta$ invariant, since one can always find such a projective transformation leaving $\Delta$ invariant and bringing the intersection of $l_1$ and $l_2$ to the center of $\Delta$ -- in this case the proposition is easy to check. 

We can now restate Pascal's theorem in terms of hyperbolic geometry. Considering the polar lines $l_1$, $l_2$ and $l_3$ of the points $X$, $Y$ and $Z$ of the statement of \Cref{thm:pascal}, we obtain the following equivalent statement.

\begin{propintro}\label{prop:hexagon_orthogonals}
	Let $ABCDEF$ be an ideal hyperbolic hexagon. 
	Let $l_1$ be the common perpendicular to $AB$ and $DE$, $l_2$ the common perpendicular to $BC$ and $EF$ and $l_3$ the common perpendicular to $CD$ and $FA$.
	Then $l_1$, $l_2$ and $l_3$ are concurrent.
\end{propintro}

\section{Proof of Proposition \ref{prop:hexagon_orthogonals}}

The proof of Proposition  \ref{prop:hexagon_orthogonals} is based on the hyperbolic version of an elementary statement on triangles.

\begin{thmintro}\label{thm:bisectors}
	Let $PQR$ be a triangle. Then the angle bisectors of $PQR$ are concurrent.
\end{thmintro}

\begin{rem*}
  \Cref{thm:bisectors} holds for Euclidean, spherical and hyperbolic triangles. The proof is, in all three cases, elementary. 
  A point of a triangle is in the bisector of an angle if and only if it is at equal distance from the two corresponding edges. Therefore, the intersection point of two angle bisectors is at equal distance from all three edges, and is therefore contained in the third bisector.
  
\end{rem*}

\begin{lemmeintro} \label{lm:quadrilateral}
	Let $ABDE$ be an ideal hyperbolic quadrilateral. 
	Then, the common orthogonal to $AB$ and $DE$ is the angle bisector of the lines $AD$ and $BE$.
\end{lemmeintro}
\begin{proof}
	Let $l$ be the angle bisector of the lines $AD$ and $BE$. The hyperbolic reflection on $l$ exchanges $A$ and $B$, so the line $AB$ is preserved by the reflection. Thus, the angle bisector of the two lines is orthogonal to $AB$. Similarly, $l$ is also perpendicular to $DE$, so it is precisely the common perpendicular to $AB$ and $DE$.
\end{proof}

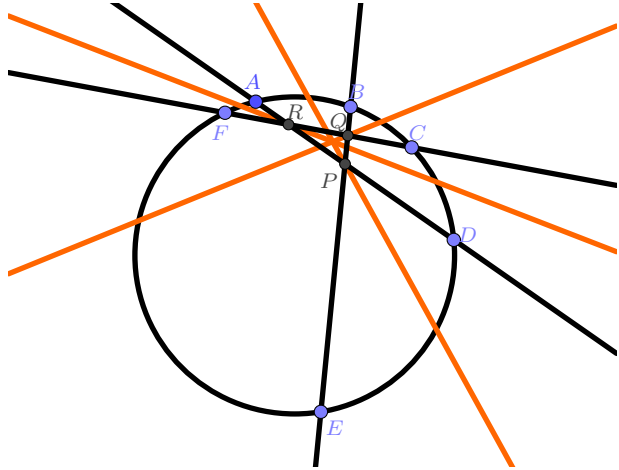
\begin{figure}
	\centering
	\definecolor{uuuuuu}{rgb}{0.26666666666666666,0.26666666666666666,0.26666666666666666}
\definecolor{ffwwqq}{rgb}{1.,0.4,0.}
\definecolor{xdxdff}{rgb}{0.49019607843137253,0.49019607843137253,1.}
\definecolor{ududff}{rgb}{0.30196078431372547,0.30196078431372547,1.}
\begin{tikzpicture}[line cap=round,line join=round,>=triangle 45,x=0.6410596026490066cm,y=0.6410596026490066cm]
\clip(-1.5760229989258046,-4.005396799771493) rectangle (10.903315844049402,5.581380059732617);
\draw [line width=2.pt] (4.3,0.36) circle (2.1020890799132075cm);
\draw [line width=2.pt,color=ffwwqq,domain=-1.5760229989258046:10.903315844049402] plot(\x,{(--20.229281814672568--13.282843146001653*\x)/32.32755488086449});
\draw [line width=2.pt,color=ffwwqq,domain=-1.5760229989258046:10.903315844049402] plot(\x,{(--26.671592364514925-2.2194503947361612*\x)/5.676689522244635});
\draw [line width=2.pt,color=ffwwqq,domain=-1.5760229989258046:10.903315844049402] plot(\x,{(--246.26175552769604-37.512633539084945*\x)/20.54566170582546});
\draw [line width=2.pt,domain=-1.5760229989258046:10.903315844049402] plot(\x,{(--24.387897562577127-2.858269780735144*\x)/4.0632636525435375});
\draw [line width=2.pt,domain=-1.5760229989258046:10.903315844049402] plot(\x,{(--32.25611599216289-6.307374040240724*\x)/-0.6093820053767534});
\draw [line width=2.pt,domain=-1.5760229989258046:10.903315844049402] plot(\x,{(--14.7364613031635-0.7164076898281753*\x)/3.8299637962245026});
\begin{scriptsize}
\draw [fill=ududff] (3.5,3.54) circle (2.5pt);
\draw[color=ududff] (3.4322414638841128,3.953280886178902) node {$A$};
\draw [fill=xdxdff] (5.445654831829739,3.4324379580888484) circle (2.5pt);
\draw[color=xdxdff] (5.564472868842791,3.738404853121051) node {$B$};
\draw [fill=xdxdff] (6.699942727554431,2.594429435998953) circle (2.5pt);
\draw[color=xdxdff] (6.820671215950229,2.895429646509483) node {$C$};
\draw [fill=xdxdff] (7.5632636525435375,0.6817302192648558) circle (2.5pt);
\draw[color=xdxdff] (7.845464604379981,0.7797271671706448) node {$D$};
\draw [fill=xdxdff] (4.836272826452985,-2.8749360821518755) circle (2.5pt);
\draw[color=xdxdff] (5.118191877107254,-3.2) node {$E$};
\draw [fill=xdxdff] (2.8699789313299284,3.310837125827128) circle (2.5pt);
\draw[color=xdxdff] (2.7710844390907243,2.9) node {$F$};
\draw [fill=uuuuuu] (5.331568880679512,2.2515977621329006) circle (2.0pt);
\draw[color=uuuuuu] (5.0,1.9) node {$P$};
\draw [fill=uuuuuu] (5.388393631118646,2.8397591331621577) circle (2.0pt);
\draw[color=uuuuuu] (5.2,3.110305679567334) node {$Q$};
\draw [fill=uuuuuu] (4.171994936126382,3.067290625717895) circle (2.0pt);
\draw[color=uuuuuu] (4.291745596115518,3.3417106382450195) node {$R$};
\end{scriptsize}
\end{tikzpicture}
	\caption{The triangle $PQR$ and its hyperbolic angle bisectors.}
\end{figure}

\begin{proof}[Proof of Proposition \ref{prop:hexagon_orthogonals}]
  If the diagonals $AD$, $BE$ and $CF$ are concurrent  at a point $P$, then \Cref{lm:quadrilateral} implies that $l_1$, $l_2$ and $l_3$ contain $P$, and are therefore concurrent.
  
  Now, suppose that the diagonals $AD$, $BE$ and $CF$ are not concurrent. We call $P$ the intersection point of $BE$ and $CF$, $Q$ the intersection point of $AD$ and $CF$, and $R$ the intersection point of $AD$ and $BE$.

  It follows from Lemma \ref{lm:quadrilateral} that the angle bisector of $PQR$ at $P$ is the common perpendicular $l_1$ to $AB$ and $DE$, while the angle bisector of $PQR$ at $Q$ is the common perpendicular $l_2$ to $BC$ and $EF$, and the angle bisector of $PQR$ at $R$ is the common perpendicular $l_3$ to $CD$ and $FA$.

  By Theorem \ref{thm:bisectors} applied to the triangle $PQR$, the lines $l_1$, $l_2$ and $l_3$ are concurrent, and the result follows.
\end{proof}

\section{The Möbius generalization}

In 1847, Möbius proved a generalization of Pascal's Theorem for $(4n+2)$-gons \cite{moebius_verallgemeinerung_1848}.

\begin{thmintro}[Möbius, 1847]
	Let $A_1 A_2 \cdots A_{4n+2}$ be a cyclic $(4n+2)$-gon. Let $X_1 , \dots , X_{2n+1}$ be the intersection points of the pairs of opposite sides of $A_1 A_2 \cdots A_{4n+2}$.
	If $X_1 , \dots , X_{2n}$ are aligned, then $X_{2n+1}$ lies in the same line as $X_1 , \dots , X_{2n}$.
\end{thmintro}

By considering the polar lines $l_1, \dots , l_{2n+1}$ of $X_1, \dots , X_{2n+1}$, we obtain the corresponding hyperbolic statement, for which the proof of \Cref{prop:hexagon_orthogonals} extends easily.

\begin{propintro}
	Let $A_1 A_2 \cdots A_{4n+2}$ be a hyperbolic ideal $(4n+2)$-gon. Let $l_1 , \dots , l_{2n+1}$ be the common perpendicular to the pairs of opposite sides of $A_1 A_2 \cdots A_{4n+2}$.
	If $l_1 , \dots , l_{2n}$ are concurrent, then the common intersection point belongs also to $l_{2n+1}$.
\end{propintro}

 For each $i \in \{ 1, \dots , 2n+1 \}$, let $m_i$ be the line joining $A_i$ and $A_{i+2n+1}$. Let $R_i$ be the be the union of the two  quarters defined by $m_i$ and $m_{i+1}$ that contain the sides $A_iA_{i+1}$ and $A_{i+2n+1}A_{i+2n+2}$, as in \Cref{fig:region_ri}, where the indices are taken modulo $2n+1$.
 Observe that, by \Cref{lm:quadrilateral}, the line $l_i$ is the set of points of $R_i$ that is at the same distance from $m_i$ and $m_{i+1}$. If $l_1 , \dots , l_{2n}$ are concurrent at a point $P$, then $P$ is at the same distance from all the lines $m_i$, so, in particular, it is at the same distance from $m_{2n+1}$ and $m_1$.
 The only remaining point to complete the proof (and the reason the statement is false for $4n$-gons) is given by the following lemma.
 
 \begin{figure}
 	\centering
 		\definecolor{ffzzqq}{rgb}{1.,0.6,0.}
	\definecolor{zzttqq}{rgb}{0.6,0.2,0.}
	\definecolor{uuuuuu}{rgb}{0.26666666666666666,0.26666666666666666,0.26666666666666666}
	\definecolor{ffwwqq}{rgb}{1.,0.4,0.}
	\definecolor{xdxdff}{rgb}{0.49019607843137253,0.49019607843137253,1.}
	\begin{tikzpicture}[line cap=round,line join=round,>=triangle 45,x=0.39919806595094903cm,y=0.39919806595094903cm]
	\clip(-1.9000097142316479,-7.32968481434277) rectangle (18.14016753698527,4.970106514446447);
	\fill[line width=2.pt,color=zzttqq,fill=zzttqq,fill opacity=0.10000000149011612] (6.079371388806722,3.829672372217824) -- (4.374759572558283,-0.6410193020411784) -- (13.546364440197735,-0.013189996996894493) -- cycle;
	\fill[line width=2.pt,color=zzttqq,fill=zzttqq,fill opacity=0.10000000149011612] (2.5127355809229748,-0.7684815341712308) -- (4.374759572558283,-0.6410193020411784) -- (3.097180558910262,-3.9917304757409093) -- cycle;
	\fill[line width=2.pt,color=zzttqq,fill=zzttqq,fill opacity=0.10000000149011612] (6.079371388806722,3.829672372217824) -- (13.546364440197735,-0.013189996996894493) -- (13.476037262076677,0.23536200172744692) -- (13.382482576357488,0.5132298959201813) -- (13.24884829438345,0.8452508120348661) -- (13.100939879506454,1.1553835290228065) -- (12.948076064849833,1.432745777217129) -- (12.720702247530516,1.7879761127019682) -- (12.48368283774907,2.1049574150653583) -- (12.192503415984184,2.439198853409576) -- (11.94534084163569,2.6855272270012414) -- (11.681938825161215,2.917247312847481) -- (11.369557185596198,3.1573350117853796) -- (11.00242891314841,3.3984962803887404) -- (10.454559553257353,3.688624992125037) -- (9.869403846641331,3.919588938063571) -- (9.60544935303871,3.999928250949023) -- (9.135208924798043,4.109045349969421) -- (8.322918931072957,4.200703439087432) -- (7.774912553908212,4.195720626277051) -- (7.170042544043759,4.12771837155703) -- (6.745715526114341,4.039662403829768) -- (6.388167867611173,3.9383235018540192) -- cycle;
	\fill[line width=2.pt,color=zzttqq,fill=zzttqq,fill opacity=0.10000000149011612] (2.5127355809229748,-0.7684815341712308) -- (2.4786562436329325,-1.2138583672496246) -- (2.482795283437949,-1.718347401921069) -- (2.5086101401135963,-2.0351092868421885) -- (2.539386931939245,-2.2694600092482013) -- (2.60664505905088,-2.6303931149630224) -- (2.6795800477505543,-2.9236780045129716) -- (2.7937683097892876,-3.2870579128145883) -- (2.9005898688843956,-3.5666565371415166) -- (3.097180558910262,-3.9917304757409093) -- cycle;
	\draw [line width=2.pt] (8.1,-1.42) circle (2.2455379113187535cm);
	\draw [line width=2.pt,color=ffwwqq,domain=-1.9000097142316479:18.14016753698527] plot(\x,{(-764.5980653353994--150.00249803487102*\x)/169.06386563493913});
	\draw [line width=2.pt] (3.097180558910262,-3.9917304757409093)-- (6.079371388806722,3.829672372217824);
	\draw [line width=2.pt] (2.5127355809229748,-0.7684815341712308)-- (13.546364440197735,-0.013189996996894493);
	\draw [line width=2.pt,color=ffzzqq] (2.539386931939245,-2.2694600092482013)-- (9.60544935303871,3.999928250949023);
	\begin{scriptsize}
	\draw [fill=xdxdff] (2.5127355809229748,-0.7684815341712308) circle (2.5pt);
	\draw[color=xdxdff] (1.5884655850542597,-0.395147289318505) node {$A_{i}$};
	\draw [fill=xdxdff] (3.097180558910262,-3.9917304757409093) circle (2.5pt);
	\draw[color=xdxdff] (1.970183246982201,-4.254736982145467) node {$A_{i+1}$};
	\draw [fill=xdxdff] (13.546364440197735,-0.013189996996894493) circle (2.5pt);
	\draw[color=xdxdff] (15,0.4319076448587009) node {$A_{i+2n+1}$};
	\draw [fill=xdxdff] (6.079371388806722,3.829672372217824) circle (2.5pt);
	\draw[color=xdxdff] (5,4.397530021554534) node {$A_{i+2n+2}$};
	\draw [fill=uuuuuu] (4.374759572558283,-0.6410193020411784) circle (2.0pt);
	\draw[color=black] (4.2,1.258962579035907) node {$m_{i+1}$};
	\draw[color=black] (9.031959992649114,-0.8828976351153186) node {$m_i$};
	\draw[color=ffzzqq] (6.317523285605977,0.4) node {$l_i$};
	\end{scriptsize}
	\end{tikzpicture}
 	\caption{The region $R_i$.} \label{fig:region_ri}
 \end{figure}
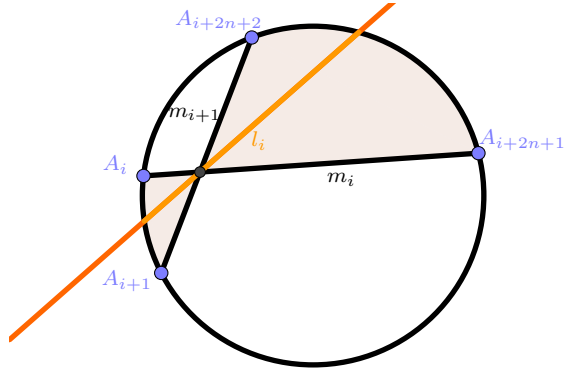
 
\begin{lemmeintro}
	If $P \in R_1 \cap \cdots \cap R_{2n}$, then $P \in R_1 \cap \cdots \cap R_{2n+1}$
\end{lemmeintro} 

\begin{proof}
	Consider Cartesian equations for the lines $m_i$, that we still denote by $m_i$, so $m_i(A_i) = m_i(A_{i+2n+1}) = 0$. Up to changing signs, we can suppose that $m_i(A_{i+1}) > 0$. 
	Thus, for each $i \in \{1, \dots ,2n \}$ the region $R_i$ is defined by the inequality $m_i m_{i+1} < 0$, but the region $R_{2n+1}$, bounded by $m_{2n+1}$ and $m_1$, is defined by $m_{2n+1}m_1 > 0$.
	Now, if $P \in R_1 \cap \cdots \cap R_{2n}$, then $m_1(P)m_2(P) < 0$, $\dots$ , $m_{2n}(P) m_{2n+1}(P) < 0$.
	By multiplying this even number of inequalities, we obtain
	$m_{2n+1}(P)m_1(P) > 0$, so $P \in R_{2n+1}$.
\end{proof}

\subsection*{Acknowledgement}

The authors are grateful to Arseniy Akopyan for pointing out reference \cite{richter-gebert_perspectives_2011}.

\bibliographystyle{plain}
\bibliography{ref_pascal.bib}

\begin{thebibliography}{1}

\bibitem{conway_pascal_2012}
John Conway and Alex Ryba.
\newblock The {{Pascal Mysticum Demystified}}.
\newblock {\em The Mathematical Intelligencer}, 34(3):4--8, September 2012.

\bibitem{conway_extending_2013}
John Conway and Alex Ryba.
\newblock Extending the {{Pascal Mysticum}}.
\newblock {\em The Mathematical Intelligencer}, 35(2):44--51, June 2013.

\bibitem{coxeter_geometry_1967}
Harold S.~M. Coxeter and S.~L. Greitzer.
\newblock {\em Geometry {{Revisited}}}.
\newblock {American Mathematical Society}, {Washington, UNITED STATES}, 1967.

\bibitem{drach_hyperbolic_2019}
Kostiantyn Drach and Richard~Evan Schwartz.
\newblock A {{Hyperbolic View}} of the {{Seven Circles Theorem}}.
\newblock {\em The Mathematical Intelligencer}, December 2019.

\bibitem{moebius_verallgemeinerung_1848}
A.~F. M{\"o}bius.
\newblock {Verallgemeinerung des Pascalschen Theorems, das in einen
  Kegelschnitt beschriebene Sechseck betreffend.}
\newblock {\em Journal f\"ur die reine und angewandte Mathematik},
  1848(36):216--220, January 1848.

\bibitem{richter-gebert_perspectives_2011}
J{\"u}rgen {Richter-Gebert}.
\newblock {\em Perspectives on {{Projective Geometry}}: {{A Guided Tour Through
  Real}} and {{Complex Geometry}}}.
\newblock {Springer Science \& Business Media}, February 2011.

\bibitem{paris1834historiettes}
G{\'e}d{\'e}on {Tallemant des R{\'e}aux}.
\newblock {\em {Les Historiettes de Tallemant des R\'eaux: M\'emoires pour
  servir \`a l'histoire du XVIIe si\`ecle}}.
\newblock {Delloye}, {Paris, France}, 1840.

\end{thebibliography}

\end{document}